\documentclass[11pt, leqno]{article}

\usepackage{graphicx}
\usepackage{subfigure}
\usepackage{amssymb}
\usepackage{amsthm}
\usepackage{amsfonts}
\usepackage{amsmath}
\usepackage{amsthm}
\usepackage{bm}             
\usepackage[mathscr]{eucal}
\usepackage{color}
\usepackage{cancel}
\usepackage[normalem]{ulem}
\usepackage{enumerate}
\usepackage{psfrag}
\usepackage{bbm}
\usepackage{appendix}
\usepackage{graphicx}
\usepackage[colorlinks=true,linkcolor=red,citecolor=blue]{hyperref}

\newcommand{\R}{\mathbb{R}}

\newcommand{\chfn}{\mathbbm{1}}
\newcommand{\bfR}{\mathbb{R}}
\newcommand{\mcE}{\mathcal{E}}
\newcommand{\mcM}{\mathcal{M}}

\newcommand{\mcV}{\mathcal{V}}
\newcommand{\mcR}{\mathcal{R}}
\newcommand{\mcU}{\mathcal{U}}
\newcommand{\bS}{S}
\newcommand{\mfR}{\mathfrak{R}}
\newcommand{\mfW}{\mathfrak{W}}
\newcommand{\mcW}{\mathcal{W}}
\newcommand{\mcL}{\mathcal{L}}

\newcommand{\Tm}{T_0}
\newcommand{\Rm}{\mcR_0}
\newcommand{\Vm}{\mcV_0}

\newtheorem{theorem}{Theorem}

\newtheorem{lemma}{Lemma}

\theoremstyle{definition}

\newtheorem{remark}{Remark}

\usepackage{fullpage}

\title{\Huge{Asymptotic Growth and Decay of Two-Dimensional Symmetric Plasmas}}
\author{Jonathan Ben-Artzi\\ {\small School of Mathematics}\\ {\small Cardiff University}\\ {\small Cardiff, United Kingdom}\\ {\small \tt Ben-ArtziJ@cardiff.ac.uk} \\[1cm]
Baptiste Morisse\\ {\small School of Mathematics}\\ {\small Cardiff University}\\ {\small Cardiff, United Kingdom}\\ {\small \tt MorisseB@cardiff.ac.uk}\\[1cm]
Stephen Pankavich \\ {\small Department of Applied Mathematics and Statistics}\\ {\small Colorado School of Mines}\\ {\small Golden, CO USA}\\ {\small \tt pankavic@mines.edu} }
\date{\today}

\begin{document}
\renewcommand{\thefootnote}{\fnsymbol{footnote}} 
\footnotetext{\emph{Acknowledgements:} JBA  and BM acknowledge support from an Engineering and Physical Sciences Research Council Fellowship (EP/N020154/1). 
SP was supported by the US National Science Foundation under awards DMS-1911145 and DMS-2107938.}     
\renewcommand{\thefootnote}{\arabic{footnote}} 
\maketitle

\begin{abstract}
We study the large time behavior of classical solutions to the two-dimensional Vlasov-Poisson (VP) and relativistic Vlasov-Poisson (RVP) systems launched by radially-symmetric initial data with compact support.
In particular, we prove that particle positions and momenta grow unbounded as $t \to \infty$ and obtain sharp rates on the maximal values of these quantities on the support of the distribution function for each system.
Furthermore, we establish nearly sharp rates of decay for the associated electric field, as well as upper and lower bounds on the decay rate of the charge density in the large time limit. We prove that, unlike (VP) in higher dimensions, smooth solutions do not scatter to their free-streaming profiles as $t \to \infty$ because nonlinear, long-range field interactions dominate the behavior of characteristics due to the exchange of energy from the potential to the kinetic term. In this way, the system may ``forget'' any previous configuration of particles.
\end{abstract}

\tableofcontents

\section{Introduction}

The motion of an electrostatic and collisionless plasma in two spatial and momentum dimensions (i.e., $x,v \in \mathbb{R}^2$) is given by the Vlasov-Poisson system:
 \begin{equation}
 \tag{VP}
 \label{VP}
\left \{ \begin{aligned}
& \partial_{t}f+v\cdot\nabla_{x}f+E \cdot\nabla_{v}f=0\\
& \rho(t,x)=\int_{\mathbb{R}^2} f(t,x,v)\,dv\\
& E(t,x) = \int_{\mathbb{R}^2} \frac{x-y}{\vert x - y \vert^2} \rho(t,y) \ dy.
\end{aligned} \right .
\end{equation}
Here, $t \geq 0$ represents time, $f(t,x,v) \geq 0$ is the particle density, $\rho(t,x)$ is the associated charge density, $E(t,x)$ is the self-consistent electric field generated by the charged particles, and we have chosen units such that the mass and charge of each particle are normalized.
In the present paper, we consider the Cauchy problem and require given initial data $f_0 \in C^1_c(\mathbb{R}^4)$ such that
$f(0,x,v) = f_0(x,v) \geq 0$
to complete the description of the system. One can also consider relativistic effects, for which the velocity, now denoted by
$$\hat{v}=\frac{v}{\sqrt{1+|v|^2}},$$
is no longer a multiple of the momentum $v$, and \eqref{VP} is replaced by the relativistic Vlasov-Poisson system
 \begin{equation}
 \tag{RVP}
 \label{RVP}
\left \{ \begin{aligned}
& \partial_{t}f+\hat{v}\cdot\nabla_{x}f+E \cdot\nabla_{v}f=0\\
& \rho(t,x)=\int_{\mathbb{R}^2} f(t,x,v)\,dv,\\
& E(t,x) = \int_{\mathbb{R}^2} \frac{x-y}{\vert x - y \vert^2} \rho(t,y) \ dy.
\end{aligned} \right .
\end{equation}
We refer to \cite{Glassey} as a general reference concerning these well-known plasma models.

It is known that given smooth initial data both \eqref{VP} and \eqref{RVP} possess smooth global-in-time solutions \cite{Rammaha, UkaiOkabe, Wollman}. 
In fact, global existence of classical solutions to the former system has also been established in three-dimensions \cite{LP, Pfaf, Schaeffer}.
Contrastingly, the unsolved problem of interest here concerns the large-time asymptotic behavior of such models.
Results of this nature exist for \eqref{VP} and \eqref{RVP} in some special cases, including the three-dimensional problem with small \cite{BD,Pausader,SPnew} or symmetric data \cite{Horst, SP, Jackcyl}, and in a one-dimensional setting \cite{BKR, GPS, GPS2, Sch}.
In general, determining the large time asymptotic behavior of the two-dimensional systems should be significantly more challenging than in higher dimensions \cite{SP2022}, as the long-range particle interactions induced by the electric field are stronger for $d < 3$ than the dispersive effects engrained within the Vlasov equation.

In the case of radially-symmetric initial data, i.e. in which $f_0$ is invariant under rotations in phase space, the solutions of both \eqref{VP} and \eqref{RVP} are known \cite{Horst2} to remain radially-symmetric for all $t \geq 0$.
In this case, it is useful to consider new variables that completely describe solutions with such symmetry.  In particular, defining the spatial radius, radial  {momentum}, and square of the angular momentum respectively by
\begin{equation}
\label{ang}
r = \vert x \vert, \qquad w = \frac{x \cdot v}{r}, \qquad \ell = \vert x \wedge v \vert^2,
\end{equation}
the radial-symmetry of $f_0$ implies that the distribution function, charge density, potential, and electric field take special forms for all time. Namely, in the classical case the particle distribution $f = f(t,r,w,\ell)$ satisfies the reduced Vlasov equation
\begin{equation}
\label{vlasovang}
\partial_{t}f+w\partial_r f+\left ( \frac{\ell}{r^3} + \frac{m(t,r)}{ 2\pi r} \right ) \partial_w f=0,
\end{equation}
whereas in the relativistic case this equation takes the form
\begin{equation}
\label{vlasovang-rel}
\partial_{t}f+\frac{w}{\sqrt{1+w^2+\ell r^{-2}}}\partial_r f+\left ( \frac{\ell r^{-3}}{\sqrt{1+w^2+\ell r^{-2}}} + \frac{m(t,r)}{2\pi r} \right ) \partial_w f=0.
\end{equation}
Here, the mass and charge density satisfy the reduced descriptions
\begin{equation}
\label{massang}
m(t,r) = 2\pi \int_0^r q \rho(t,q) \ dq
\end{equation}
and
\begin{equation}
\label{rhoang}
\rho(t,r) = \frac{1}{r} \int_0^\infty \int_{-\infty}^{\infty} f(t,r,w,\ell) \ell^{-1/2} \ dw \ d\ell.
\end{equation}
The electric field is then given by the expression
\begin{equation}
\label{fieldang}
E(t,x) = \frac{m(t,r)}{2\pi r} \frac{x}{r},
\end{equation}
and the associated electric potential
\begin{equation}
\label{U}
\begin{split}
\mcU(t,r) 
&=-\frac{1}{2\pi}\ln |x| \star\rho(t,x)\\\
&= -\frac{1}{2\pi} \int_0^r \frac{m(t,q)}{q} \ dq - \int_0^\infty\rho(t,q)q\ln(q)\ dq
\end{split}
\end{equation}
satisfies
$$-\nabla_x \mcU(t,r) = \frac{m(t,r)}{2 \pi r} \frac{x}{r} = E(t,x).$$
For completeness, a full derivation of these representations is given in the appendix.

Notice that while the symmetry does not significantly alter the complexity of the Vlasov equation (i.e., phase space is described by three independent variables rather than four),
the form of the resulting electric field is considerably simpler and will allow us to easily orient the force imposed on particles with respect to the origin.
The total mass is conserved and can be expressed as
$$\mcM = 2\pi \int_0^\infty \int_{-\infty}^{\infty} \int_0^\infty f_0(r, w, \ell) \ell^{-1/2} \ d\ell dw dr$$
so that $0 \leq m(t,r) \leq \mcM$ for all $t \geq 0$ and $r > 0$.
Finally, the energy of either system is conserved in time and can be written as the sum of the kinetic and potential parts:
\begin{eqnarray*}
\mcE_{\mathrm{VP}} & = &\frac{1}{2} \int_0^\infty \int_{-\infty}^\infty \int_0^\infty (w^2 + \ell r^{-2} ) f(t,r,w,\ell) \ell^{-1/2} \ d\ell dw dr\\
& \ & \quad  + \frac{1}{2}\int_0^\infty \int_{-\infty}^\infty \int_0^\infty \mcU(t,r) f(t,r,w,\ell) \ell^{-1/2} \ d\ell dw dr
\end{eqnarray*}
for \eqref{VP}, and 
\begin{eqnarray*}
\mcE_{\mathrm{RVP}} & = &\int_0^\infty \int_{-\infty}^\infty \int_0^\infty \sqrt{1+w^2 + \ell r^{-2}} f(t,r,w,\ell) \ell^{-1/2} \ d\ell dw dr\\
& \ & \quad  + \frac{1}{2}\int_0^\infty \int_{-\infty}^\infty \int_0^\infty \mcU(t,r) f(t,r,w,\ell) \ell^{-1/2} \ d\ell dw dr
\end{eqnarray*}
for \eqref{RVP}. Though the sum of the kinetic and potential energies balances at all times, we will show that each quantity actually tends to infinity with rate $\mathcal{O}(\ln(t))$ as $t \to \infty$ (see also \cite{Dolbeault}, \cite{DolRein}).
This feature will be a crucial mechanism in establishing rates for the large time behavior of the maximal position and  {momentum} on the support of the solution. Note that it is not \emph{a priori} obvious that the energies are, in fact, finite. However, we will assume (see \eqref{A} below) that the angular momenta of particles are uniformly bounded below on the support of the distribution function, and this condition ensures finite energy, as mentioned in the forthcoming Remark \ref{rek:integrability-of-energies}.
%

In the angular coordinates described above, the characteristics of the Vlasov equation also assume a reduced form.
In particular, for \eqref{VP} these are
\begin{equation}
\label{charang}
\left \{
\begin{aligned}
&\dot{\mcR}(s)=\mcW(s),\\
&\dot{\mcW}(s)= \frac{\mcL(s)}{\mcR(s)^3} + \frac{m(s, \mcR(s))}{2\pi\mcR(s)},\\
&\dot{\mcL}(s)= 0,
\end{aligned}
\right.
\end{equation}
while for \eqref{RVP} they are
\begin{equation}
\label{charang-rel}
\left \{
\begin{aligned}
&\dot{\mcR}(s)=\frac{\mcW(s)}{\sqrt{1+\mcW(s)^2+\mcL(s)\mcR(s)^{-2}}},\\
&\dot{\mcW}(s)= \frac{\mcL(s)\mcR(s)^{-3}}{\sqrt{1+\mcW(s)^2+\mcL(s)\mcR(s)^{-2}}} + \frac{m(s, \mcR(s))}{2\pi\mcR(s)},\\
&\dot{\mcL}(s)= 0.
\end{aligned}
\right.
\end{equation}
We will study forward characteristics of these systems with initial conditions
\begin{equation}
\label{charanginit}
\mcR(0) = r, \qquad \mcW(0) = w, \qquad \mcL(0) = \ell
\end{equation}
and note that the traditional convention for notation has been shortened so that
$$\mcR(s) = \mcR(s,0,r,w,\ell), \qquad \mcW(s) = \mcW(s,0,r,w,\ell), \qquad \mcL(s) = \mcL(s,0,r,w,\ell).$$
Additionally, because the angular momentum of particles is conserved in time on the support of $f(t)$, we note that $\mcL(s) = \ell$ for every $s \geq 0$.

Though we will focus on two-dimensional problems, it is useful to note that such solutions can also satisfy analogous three-dimensional problems. In particular, if one prescribes initial data for \eqref{VP} or \eqref{RVP} with $x,v\in \bfR^3$ that is cylindrically symmetric \cite{Jackcyl}, independent of $x_3$, and possesses a Dirac delta dependence on $v_3$, then any solution of \eqref{vlasovang} or \eqref{vlasovang-rel}, respectively, will automatically satisfy these equations in the sense of distributions.

In establishing the forthcoming results, we face some challenging issues.
As we will show, the methods used to understand the behavior of solutions in three-dimensions cannot obtain sharp rates for the two-dimensional \eqref{VP} problem, as the energy transfer in the latter system is the driving force for the behavior of characteristics. 
An additional issue is that the supremum of the field decays very slowly in time,
and thus no convergence of  {momentum} characteristics can be expected.

In order to precisely state the main results, we first define notation for the (interior) support of $f$ and the maximal particle position and radial momentum on this set. For every $t \geq 0$, define
$$S(t) = \left \{ (r, w, \ell) : f(t, r, w, \ell) > 0 \right \},$$
as well as
$$\mfR(t) =\sup_{(r, w, \ell) \in \bS(0)} \mcR(t, 0, r, w, \ell),$$
and
$$\mfW(t) =\sup_{(r, w, \ell) \in \bS(0)} \left | \mcW(t, 0, r, w, \ell) \right |.$$
We further define the projection
$$\pi_r(\bS(t)) = \{r  : (r,w,\ell) \in \bS(t)\}$$
with analogous notation for $\pi_w$ and $\pi_\ell$.
Additionally, we use the notation
$A(t) \lesssim B(t)$
to mean that there is $C > 0$ {, independent of $t$,} such that
$$ A(t) \leq C B(t)$$%
for all $t$ sufficiently large
with an analogous definition for $\gtrsim$,
and
$A(t) \sim B(t)$
to mean 
$$B(t) \lesssim A(t) \lesssim B(t).$$
Throughout we will assume that all particles possess some positive angular momentum on the support of $f$, namely there is $C> 0$ such that
\begin{equation}
\label{A}
\tag{A}
\inf \pi_\ell(\bS(0)) = \inf \{ \ell : (r,w,\ell) \in \bS(0)\} \geq C.
\end{equation}
\begin{remark}\label{rek:integrability-of-energies}
Note that the compact support of $f_0$ and  \eqref{A} guarantee that the potential and kinetic energies are finite for both \eqref{VP} and \eqref{RVP}. Indeed, these energies involve the term $\ell^{-1/2}$, and the kinetic energy has the term $r^{-2}$ in addition. The compact support of $f_0$ guarantees that the support of $f(t)$ remains compact for all times so that there are no issues of integrability  at infinity. Moreover,  \eqref{A} guarantees that the support of $f_0$ is bounded away from both $r=0$ and $\ell=0$, and this remains true for the support of $f(t)$ at later times.
\end{remark}

With this, we prove decay rates for the field, as well as sharp growth rates for the maximal particle momenta and positions for each system.
Here, the leading order dynamics of the particle characteristics for \eqref{VP} are not driven merely by their angular momentum, which is the case for the 3D spherically-symmetric problem, but also by the transfer from potential energy to kinetic energy.
Furthermore, the asymptotic rates attained by solutions of \eqref{VP} and \eqref{RVP} differ from one another, unlike in the three-dimensional case.
\begin{theorem}
\label{T1}
Let $f_0 \in C^1(\mathbb{R}^4)$ be nontrivial and radially-symmetric with compact support satisfying \eqref{A}, and let $p \in (2,\infty]$. Then, for any solution of \eqref{VP}, we have
$$ \begin{gathered}
\mfW(t) \sim\sqrt{\ln(t)},\\
\mfR(t) \sim t\sqrt{\ln(t)},\\
\Vert \mcU(t) \Vert_\infty \sim \ln(t),\\
\end{gathered}$$
as well as the field and density estimates
$$ \begin{gathered}
\left (t \sqrt{\ln(t)} \right)^{-1 + \frac{2}{p}} \lesssim \Vert E(t) \Vert_p \lesssim t^{-1 + \frac{2}{p}}, \\
\left (t^2\ln(t) \right)^{-1} \lesssim \Vert \rho(t) \Vert_\infty \lesssim t^{-1},
\end{gathered}$$
and the pointwise estimates
$$ \begin{gathered}
0 \lesssim \mcW(t, 0, r, w, \ell)  \lesssim \sqrt{\ln(t)},\\
t \lesssim \mcR(t, 0, r, w, \ell) \lesssim t\sqrt{\ln(t)}\\
\end{gathered}
$$
for $(r,w,\ell) \in S(0)$.
\end{theorem}
\begin{theorem}
\label{T2}
Let $f_0 \in C^1(\mathbb{R}^4)$ be nontrivial and radially-symmetric with compact support satisfying \eqref{A}, and let $p \in (2,\infty]$. Then, for any solution of \eqref{RVP} we have
$$ \begin{gathered}
\mfW(t) \sim \ln(t),\\
\mfR(t) \sim t,\\
\Vert \mcU(t) \Vert_\infty \sim \ln(t),
\end{gathered}$$
as well as the field and density estimates
$$ \begin{gathered}
\Vert E(t) \Vert_p \sim t^{-1 + \frac{2}{p}},\\
t^{-2} \lesssim \Vert \rho(t) \Vert_\infty \lesssim t^{-1}
\end{gathered}$$
and the pointwise estimates
$$ \begin{gathered}
0 \lesssim \mcW(t, 0, r, w, \ell)  \lesssim \ln(t),\\
\mcR(t, 0, r, w, \ell) \sim t
\end{gathered}
$$
for $(r,w,\ell) \in S(0)$.
\end{theorem}

The reader will note that we do not obtain sharp \emph{pointwise} estimates of the positions and momenta, with the exception of $\mcR(t)$ in Theorem \ref{T2}. Hence, it is possible that not all particles asymptotically disperse at the same rate. Indeed, equations \eqref{V2} and \eqref{eq:R''-rel}, which will be presented later, show that particles which continually experience a nontrivial force (i.e., $m(t,\mcR(t)) \geq C > 0$) will disperse at the greater asymptotic rates, while it is possible that those particles which experience arbitrarily small forces will instead have positions and momenta that grow at lesser rates.
For this reason, it remains an open problem to either demonstrate the multiple asymptotic dynamics of characteristics or obtain the sharp lower bound.
%
Additionally, we note that the contribution of the electric field  {can} dominate the influence of dispersive effects in the asymptotic behavior of characteristics. 
Indeed, for characteristics satisfying $m(t,\mcR(t)) \gtrsim 1$, we have
$$ \left | \mcW(t,\tau, r,w, \ell) - w \right | = \int_\tau^t \left ( \frac{m(s,\mcR(s))}{\mcR(s)} + \ell \mcR(s)^{-3} \right ) \ ds \geq C\int_\tau^t \left (s \sqrt{\ln(s)} \right)^{-1}  ds \gtrsim \sqrt{\ln(t)}$$
for \eqref{charang} and similarly
$$ \left | \mcW(t,\tau, r,w, \ell) - w \right |  \geq C\int_\tau^t s^{-1}  ds \gtrsim \ln(t)$$
for \eqref{charang-rel}
by taking $\tau$ sufficiently large and $(r,w,\ell) \in S(\tau)$.
Thus, either system may contain particles such that $\mcR(t)$ and $\mcW(t)$ ultimately ``forget'' their values at any previous time,
and  {these momentum} characteristics cannot converge as $t \to \infty$.
An analogous calculation for characteristics of \eqref{VP} further yields
$$\left | \mcR(t, \tau, r, w, \ell) - \left ( r + w t\right ) \right | \gtrsim t \sqrt{\ln(t)}$$
for $\tau$ sufficiently large and $(r,w,\ell) \in S(\tau)$.
 {In particular, we note that the distance between the spatial characteristics and their free-streaming counterparts is growing faster than the free-streaming trajectories themselves as $t \to \infty$.}
Therefore, unlike solutions of the three-dimensional problem \cite{Pausader, SP, SPnew}  {in which this difference is lower order},   {it is not clear if one can obtain} modified convergence of the particle distribution or its spatial average as $t \to \infty$.  {We therefore present this as an interesting open problem that remains elusive, though we conjecture that the distribution of the angular momentum, which is time-independent, is the only microscopic information retained in the time-asymptotic limit.}\\

The paper proceeds as follows. 
The proofs of Theorems~\ref{T1} and \ref{T2} are contained within Section~\ref{thmproof}.
Section~\ref{estimates1} is devoted to obtaining preliminary estimates for the particle characteristics, potential, and electric field, some of which are further improved in Section~\ref{estimates2} using the growth of the kinetic energy.
The charge density is then estimated in Section~\ref{density}.
Within the proofs we inherently assume $f_0 \in C_c^1(\bfR^4)$ is nontrivial and radially symmetric, and note that $C$ will represent a constant that may change from line to line, but when necessary to denote a certain constant, we will distinguish this value with a subscript, e.g. $C_0$.  {As mentioned in the discussion, all theorems herein pertain only to the large time behavior of solutions.}

\section{Estimates of the Characteristics, the Field, and the Potential}
\label{estimates1}
In this section, we state and prove a variety of lemmas concerning the behavior of particle characteristics, the potential, and the electric field.   {An important quantity here and in the sequel shall be $w^2+\ell r^{-2}$, which is simply the representation of $|v|^2$ in the aforementioned coordinates (see \eqref{eq:expression-for-v} in  Appendix \ref{appendix} and the surrounding discussion).}

\subsection{Behavior of Characteristics}
We first study the behavior of the characteristics \eqref{charang} and \eqref{charang-rel} corresponding to the classical and relativistic systems, respectively.
Some of the ideas here are derived from the three-dimensional problem with spherical symmetry \cite{BCP1, BCP2, Horst, SP}.
 {The repulsive force is crucial for our methods, as it guarantees that particles only experience forces that push them away from the origin.
Indeed, in the attractive case, steady state solutions are known to exist, and the particles need not disperse.}

\begin{lemma}
\label{L1}
Let $r, \ell > 0$ and $w \in \mathbb{R}$ be given, and let $(\mcR(t), \mcW(t), \ell)$ satisfy either \eqref{charang} or \eqref{charang-rel} for  all $t \geq 0$, with initial conditions as in  \eqref{charanginit}.
Then, we have the following:
\begin{enumerate}
\item For solutions of \eqref{charang}
$$\mcR(t)^2 \geq \ell r^{-2}t^2$$
for every $t \geq 0$, while for solutions of \eqref{charang-rel}
$$\mcR(t)^2 \geq \frac{\ell r^{-2}}{1 + w^2 + \ell r^{-2}}t^2$$
for every $t \geq 0$.

\item There is $C > 0$ such that for any $(r,w,\ell) \in S(0)$, we have
$$\mcR(t)^2 \geq Ct^2$$
for every $t \geq 0$.

\item There exists a ``turn-around time'' $T= T(r, w, \ell) \geq 0$ such that
$$\mcW(t) = \dot{\mcR}(t) > 0$$
for all $t \in (T, \infty)$.
Furthermore, for both solutions of  \eqref{charang} and \eqref{charang-rel} it holds that $T = 0$ if $w \geq 0$. If $w<0$, then  for solutions of \eqref{charang}
$$ 0 < T \leq \frac{|w| r^3}{\ell},$$
while for solutions of \eqref{charang-rel}
$$ 0 < T \leq \frac{|w| r^3\sqrt{1+w^2+\ell r^{-2}}}{\ell}.$$
\end{enumerate}
\end{lemma}

\begin{proof}
We first prove the result for  the characteristics \eqref{charang} of \eqref{VP} and then for the characteristics \eqref{charang-rel} of \eqref{RVP}.
To begin, we note the convexity of the spatial characteristics.  In particular, we find
\begin{equation}
\label{R2}
\frac{d^2}{dt^2} \left ( \mcR(t)^2 \right ) = 2 \left (\mcW(t)^2 + \ell \mcR(t)^{-2} \right ) + \frac{1}{\pi} m(t, \mcR(t)) \geq 2 \left (\mcW(t)^2 + \ell \mcR(t)^{-2} \right ).
\end{equation}
Similarly, the  {momentum} characteristics satisfy
\begin{equation}
\label{Winc}
\dot{\mcW}(t) \geq \ell \mcR(t)^{-3} > 0,
\end{equation}
and thus $\mcW(t)$ is increasing for $t \in [0,\infty).$
Finally, the square of the  {momentum} magnitude satisfies
\begin{equation}
\label{V2}
\frac{d}{dt} \left ( \mcW(t)^2 + \ell \mcR(t)^{-2} \right ) = \frac{m(t, \mcR(t))}{\pi\mcR(t)} \mcW(t).
\end{equation}

We first consider the case $w \geq 0$. Then, by \eqref{Winc} it follows that $\mcW(t) > w \geq 0$ for all $t \geq 0$.
The identity \eqref{V2} together with the positivity of the mass both imply that
$$\frac{d}{dt} \left ( \mcW(t)^2 + \ell \mcR(t)^{-2} \right ) \geq 0$$
for all $t \geq 0$, and because this function is increasing, \eqref{R2} then yields
$$\frac{d^2}{dt^2} \left ( \mcR(t)^2 \right ) \geq 2(w^2 + \ell r^{-2}).$$
Integrating in $t$ twice then implies
$$\mcR(t)^2 \geq r^2 + 2rwt + \left (w^2 + \ell r^{-2} \right )t^2 \geq \ell r^{-2} t^2$$
which provides the stated lower bound.

Now, instead consider the case $w < 0$. Then, define the ``turn-around'' time
	\begin{equation}\label{eq:turn-around}
	\Tm = \sup \{t \geq 0 : \mcW(t) \leq 0 \}
	\end{equation}
and note that $w < 0$ implies that $\Tm > 0$.
We first show that $\Tm < \infty$. For the sake of contradiction, assume $\Tm =\infty$. Then, we have $\dot{\mcR}(t) = \mcW(t) \leq 0$ for all $t \geq 0$ and thus $\mcR(t) \leq r$ for all $t \geq 0$.  From \eqref{charang} and the nonnegativity of the mass, we find
$$ \ddot{\mcR}(t)  = \frac{\ell}{\mcR(t)^3} + \frac{m(t, \mcR(t))}{2\pi\mcR(t)}\geq \ell \mcR(t)^{-3} \geq \ell r^{-3},$$
and upon integrating this yields
$$\mcW(t) \geq \ell r^{-3}t + w$$
for all $t \geq 0$.
Taking $t > \frac{-w r^3}{\ell }$ implies $\mcW(t) > 0$, thus contradicting the assumption that $\Tm = \infty$, and we conclude that $\Tm$ must be finite.
With this, the upper bound $T_0 \leq \frac{-w r^3}{\ell }$ follows, as well.

Since $\Tm < \infty$ and $\dot{\mcW}(t) > 0$ for all $t \geq 0$, we find
$$
\left \{
\begin{gathered}
\mcW(t)  < 0 \ \mathrm{for} \ t \in [0,\Tm),\\
\mcW(\Tm)  =0, \ \mathrm{and}\\
\mcW(t)  > 0 \ \mathrm{for} \ t \in (\Tm,\infty).
\end{gathered}
\right.$$
With this, we proceed as in the $w \geq 0$ case, but on the interval $[\Tm, \infty)$. In particular, \eqref{V2} shows that
$$\frac{d}{dt} \left ( \mcW(t)^2 + \ell \mcR(t)^{-2} \right ) \biggr \vert_{t = \Tm} = \frac{m(\Tm, \mcR(\Tm))}{\pi\mcR(\Tm)} \mcW(\Tm) = 0$$
and implies that $\mcW(t)^2 + \ell \mcR(t)^{-2}$ is minimized at $\Tm$, as the derivative changes from negative to positive at $t = \Tm$.
Thus, we define
$$\Rm^2 := \min_{t \geq 0} \mcR(t)^2 = \mcR(\Tm)^2$$
and
$$\Vm^2 := \min_{t \geq 0} \left ( \mcW(t)^2 + \ell \mcR(t)^{-2} \right )= \ell \mcR(\Tm)^{-2}.$$
The identity
\begin{equation}
\label{lrep}
\ell = \Rm^2 \Vm^2
\end{equation}
then follows immediately.

Now, using equation \eqref{R2} we find
$$\frac{d^2}{dt^2} ( \mcR(t)^2 ) \geq 2  \left ( \mcW(t)^2 + \ell \mcR(t)^{-2} \right ) \geq 2\Vm^2$$
for all $t \geq 0$.
Integrating twice in time yields
\begin{eqnarray*}
\mcR(t)^2 & \geq & \mcR(\Tm)^2 + 2\mcR(\Tm)\mcW(\Tm) (t - \Tm) + \Vm^2 (t - \Tm)^2\\
& = & \Rm^2 + \Vm^2 (t - \Tm)^2
\end{eqnarray*}
for any $t \geq 0$. In particular, evaluating this expression at $t = 0$ gives
\begin{equation}
\label{rgeq}
r^2 \geq \Rm^2 + \Vm^2 \Tm^2.
\end{equation}
Returning to the original lower bound for $\mcR(t)^2$, we divide by $t^2$ to find
$$\frac{\mcR(t)^2}{t^2} \geq \frac{\Rm^2 + \Vm^2 (t - \Tm)^2}{t^2}.$$
The right side of this inequality can then be minimized over all $t > 0$ and we find
$$\frac{\Rm^2 + \Vm^2 (t - \Tm)^2}{t^2} \geq \frac{\Rm^2 \Vm^2}{\Rm^2 + \Vm^2 \Tm^2},$$
which then yields the lower bound
\begin{equation}
\label{R2lower}
\mcR(t)^2 \geq \frac{\Rm^2 \Vm^2}{\Rm^2 + \Vm^2 \Tm^2} t^2.
\end{equation}
Using \eqref{lrep} with \eqref{rgeq} in \eqref{R2lower} yields
$$\mcR(t)^2 \geq \frac{\ell}{\Rm^2 + \Vm^2 \Tm^2} t^2 \geq \ell r^{-2}t^2$$
and the desired lower bound is again achieved.
Because this occurs in both cases, the proof of the first result for \eqref{VP} is complete.

The second result merely follows from the compactness of the set $S(0)$ and \eqref{A}.
Indeed, for $(r, w, \ell) \in S(0)$, there is $C > 0$ such that $r \leq C$, and using this lower bound within the first result yields
$$\mcR(t)^2 \geq C\ell t^2 \geq C t^2.$$
Finally, the last result merely follows from the previous argument with $T = 0$ if $w \geq 0$ and $T = \Tm$ if $w < 0$.

Next, we establish the stated results for solutions of \eqref{charang-rel} using similar methods. We first observe that in the relativistic case
	\begin{align}
	\frac12\frac{d^2}{dt^2}(\mcR(t)^2)
	&=
	\frac{\ell \mcR(t)^{-2}m(t,\mcR(t))}{2\pi\left(1+\mcW(t)^2+\ell\mcR(t)^{-2}\right)^{3/2}}
	+
	\frac{\ell\mcR(t)^{-2}+\mcW(t)^2}{1+\mcW(t)^2+\ell\mcR(t)^{-2}}
	+
	\frac{m(t,\mcR(t))}{2\pi\left(1+\mcW(t)^2+\ell\mcR(t)^{-2}\right)^{1/2}}\notag\\
	&\geq
	\frac{\ell\mcR(t)^{-2}+\mcW(t)^2}{1+\mcW(t)^2+\ell\mcR(t)^{-2}}
	\geq0\label{eq:R''-rel}
	\end{align}
and
	\begin{equation*}
	\dot{\mcW}(t)= \frac{\ell\mcR(t)^{-3}}{\sqrt{1+\mcW(t)^2+\ell\mcR(t)^{-2}}} + \frac{m(t, \mcR(t))}{2\pi\mcR(t)}>0.
	\end{equation*}
Therefore, $\mcW(t)$ is increasing for all $t\geq0$. Another essential quantity is the derivative of the rest  {momentum}, namely
	\begin{equation}\label{eq:change-in-rest-vel}
	\frac{d}{dt}\sqrt{1+\mcW(t)^2+\ell\mcR(t)^{-2}}
	=
	\frac{m(t,\mcR(t))\mcR(t)^{-1}}{\sqrt{1+\mcW(t)^2+\ell\mcR(t)^{-2}}} \mcW(t)
	\end{equation}
the sign of which depends on $\mcW(t)$.

We first consider the case $w\geq0$. Then, since $\dot\mcW(t))>0$, we know $\mcW(t)>w\geq0$ so that the derivative of the rest  {momentum} is nonnegative. This  implies
	\begin{equation*}
	\mcW(t)^2+\ell\mcR(t)^{-2}
	\geq
	w^2+\ell r^{-2}
	\end{equation*}
for all $t\geq0$. As a consequence, we deduce
	\begin{equation*}
	\frac{\mcW(t)^2+\ell\mcR(t)^{-2}}{1+\mcW(t)^2+\ell\mcR(t)^{-2}}
	\geq
	\frac{w^2+\ell r^{-2}}{1+w^2+\ell r^{-2}}
	>0
	\end{equation*}
for all $t\geq0$, which in view of \eqref{eq:R''-rel}, provides a lower bound for $\frac{d^2}{dt^2}(\mcR(t)^2)$ that is independent of $t$. Therefore, integrating twice in time and using the initial conditions yields
	\begin{align*}
	\mcR(t)^2
	&\geq
	\frac{w^2+\ell r^{-2}}{1+w^2+\ell r^{-2}}t^2+\frac{2rw}{\sqrt{1+w^2+\ell r^{-2}}} t+r^2.
	\end{align*}
This lower bound is a perfect square, which we write as
	\begin{align*}
	\mcR(t)^2
	\geq
	\frac{\ell r^{-2}}{1+w^2+\ell r^{-2}}t^2
	+
	\left(r+\frac{w}{\sqrt{1+w^2+\ell r^{-2}}}t\right)^2
	\geq
	\frac{\ell r^{-2}}{1+w^2+\ell r^{-2}}t^2,
	\end{align*}
and this is the desired lower bound for \eqref{RVP}. 

In the case $w < 0$, we again aim to show that the ``turn-around'' time $\Tm$ as defined in \eqref{eq:turn-around} is finite.
This is shown via contradiction as before; we outline the proof for completeness. If $\Tm =\infty$, then  $\mcW(t)\leq0$ for all $t\geq0$ so that $\mcR(t)\leq r$ for all $t\geq 0$, as $\dot\mcR(t)$ and $\mcW(t)$ have the same sign. From \eqref{eq:change-in-rest-vel} we find 
	\begin{equation*}
	\mcW(t)^2+\ell\mcR(t)^{-2}
	\leq
	w^2+\ell r^{-2}
	\end{equation*}
for all $t\geq0$. Using the expression \eqref{charang-rel} for $\dot\mcW(t)$ we have
	\begin{align*}
	\dot{\mcW}(t)
	&=
	\frac{\ell\mcR(t)^{-3}}{\sqrt{1+\mcW(t)^2+\ell\mcR(t)^{-2}}} + \frac{m(t, \mcR(t))}{2\pi\mcR(t)}\\
	&\geq
	\frac{\ell\mcR(t)^{-3}}{\sqrt{1+\mcW(t)^2+\ell\mcR(t)^{-2}}} 
	\geq
	\frac{\ell}{r^3\sqrt{1+w^2+\ell r^{-2}}}.
	\end{align*}
A simple integration in time leads to
	\begin{equation*}
	\mcW(t)
	\geq
	\frac{\ell t}{r^3\sqrt{1+w^2+\ell r^{-2}}}+w
	\end{equation*}
for all $t\geq0$, and we immediately observe that for any $t>\frac{-wr^3\sqrt{1+w^2+\ell r^{-2}}}{\ell}$ we have  $\mcW(t)>0$, contradicting the assumption that $T_0=\infty$. Thus, we conclude that $T_0<\infty$, and moreover, the preceding argument leads us to the bound
	\begin{equation*}
	T_0\leq\frac{-wr^3\sqrt{1+w^2+\ell r^{-2}}}{\ell}.
	\end{equation*}
Since $\Tm < \infty$ and $\dot{\mcW}(t) > 0$ for all $t \geq 0$, we find
$$
\left \{
\begin{gathered}
\mcW(t)  < 0 \ \mathrm{for} \ t \in [0,\Tm),\\
\mcW(\Tm)  =0, \ \mathrm{and}\\
\mcW(t)  > 0 \ \mathrm{for} \ t \in (\Tm,\infty).
\end{gathered}
\right.$$
Now, the preceding analysis which relied on $w\geq0$ can be reproduced, only for times $t\geq\Tm$. From \eqref{eq:change-in-rest-vel} we have
	\begin{equation*}
	\frac{d}{dt}\sqrt{1+\mcW(t)^2+\ell\mcR(t)^{-2}} \biggr \vert_{t = \Tm}
	=
	\frac{\mcW(\Tm)m(\Tm,\mcR(\Tm))\mcR(\Tm)^{-1}}{2\pi\sqrt{1+\mcW(\Tm)^2+\ell\mcR(\Tm)^{-2}}}
	=
	0,
	\end{equation*}
and this implies that both $\mcR(t)^2$ and $\mcW(t)^2+\ell\mcR(t)^{-2}$ are minimized at $\Tm$, as their respective derivatives change from negative to positive at $t = \Tm$.
Thus, we define
$$\Rm^2 := \min_{t \geq 0} \mcR(t)^2 = \mcR(\Tm)^2$$
and
$$\Vm^2 := \min_{t \geq 0}\left(\mcW(t)^2+\ell\mcR(t)^{-2}\right)= \ell \mcR(\Tm)^{-2}.$$
For brevity we also define $\hat\mcV_0:=\mcV_0/\sqrt{1+\mcV_0^2}$. Then the following inequality holds:
\begin{equation}
\label{lrep-rel}
\Rm^2 \hat\mcV_0^2=\frac{\Rm^2\mcV_0^2}{1+\mcV_0^2}=\frac{\ell}{1+\mcV_0^2}\geq\frac{\ell}{1+w^2+\ell r^{-2}}.
\end{equation}
From \eqref{eq:R''-rel} it follows that
	\begin{align*}
	\frac12\frac{d^2}{dt^2}(\mcR(t)^2)
	&\geq
	\frac{\ell\mcR(t)^{-2}+\mcW(t)^2}{1+\mcW(t)^2+\ell\mcR(t)^{-2}}
	\geq
	\hat\mcV_0^2
	\end{align*}
for all $t\geq0$, which leads to
	\begin{equation}\label{eq:R-rel-lower-bound-1}
	\mcR(t)^2\geq\mcR_0^2+\hat\mcV_0^2(t-T_0)^2
	\end{equation}
for all $t\geq0$ since $\mcW(T_0)=0$. Evaluating at $t=0$ one obtains
	\begin{equation}\label{eq:r-rel-lower-bound}
	r^2\geq \mcR_0^2+\hat\mcV_0^2T_0^2.
	\end{equation}
Since our goal is to bound $\mcR(t)$ from below by $t^2$, we consider \eqref{eq:R-rel-lower-bound-1} divided by $t^2$:
	\begin{equation*}
	\frac{\mcR(t)^2}{t^2}\geq\frac{\mcR_0^2+\hat\mcV_0^2(t-T_0)^2}{t^2}
	\end{equation*}
which holds for any $t>0$. Now, the right side can be minimized over $t>0$ (by simply taking its derivative and determining roots). One finds
	\begin{equation*}
	\frac{\mcR_0^2+\hat\mcV_0^2(t-T_0)^2}{t^2}
	\geq
	\frac{\mcR_0^2\hat\mcV_0^2}{\mcR_0^2+\hat\mcV_0^2T_0^2}
	\end{equation*}
and we therefore obtain
	\begin{equation}\label{eq:R-rel-lower-bound-2}
	\mcR(t)^2
	\geq
	\frac{\mcR_0^2\hat\mcV_0^2}{\mcR_0^2+\hat\mcV_0^2T_0^2}t^2.
	\end{equation}
Plugging into \eqref{eq:R-rel-lower-bound-2} the bounds \eqref{lrep-rel} and  \eqref{eq:r-rel-lower-bound}, we find
	\begin{equation*}
	\mcR(t)^2
	\geq
	\frac{\ell r^{-2}}{1+w^2+\ell r^{-2}}t^2
	\end{equation*}
as required. 
The second result follows, as before, from the compactness of the set $S(0)$ and the angular momentum assumption \eqref{A}. 
Indeed, for $(r, w, \ell) \in S(0)$, there is $C > 0$ such that 
$$r^2(1 + w^2) + \ell  \leq C,$$
and using this bound within the first result yields
$$\mcR(t)^2 \geq C\ell t^2 \geq Ct^2.$$
Finally, the statement about the ``turn-around'' time is again obtained by setting $T = 0$ if $w \geq 0$ and $T = \Tm$ if $w < 0$.
\end{proof}

\subsection{Field and Potential Estimates}
We can now estimate the electric field in $L^\infty$ (Lemma \ref{L2}) and $L^p$ for $p> 2$ (Lemma \ref{Ep}), and then determine the asymptotic behavior of the potential $\mcU$ (Lemma \ref{lem:U-decay}). The only nontrivial element from the prequel which is required is the bound
		\begin{equation*}
	\mcR(t,0,r,w,\ell)^2
	\geq
	Ct^2,
	\quad
	\forall(r,w,\ell)\in S(0)
	\end{equation*}
for $t \geq 0$. 
As this was established for both \eqref{VP} and \eqref{RVP}, the estimates here hold for both systems as well. 

\begin{lemma}
\label{L2}
For solutions of \eqref{VP} or \eqref{RVP} 
we have
	\begin{equation}\label{eq:field-bounds}
	\mfR(t)^{-1}\lesssim\Vert E(t) \Vert_\infty \lesssim t^{-1}.
	\end{equation}
\end{lemma}

\begin{proof}
We first show the upper bound and begin by estimating the enclosed mass. The Vlasov equation implies that for every $t \geq 0$ and $(r,w,\ell) \in \bS(t)$
\begin{equation}
\label{Vf}
f(t,r, w, \ell) = f_0(\mcR(0, t, r, w, \ell), \mcW(0, t, r, w, \ell), \ell).
\end{equation}
Hence, we find for any $R > 0$
\begin{eqnarray*}
m(t,R) & = & 2\pi \int_0^R \int_{-\infty}^\infty \int_0^\infty f(t,r, w, \ell) \ell^{-1/2} \ d\ell dw dr\\
& = & 2\pi \iiint\limits_{\bS(t)} f(t,r, w, \ell) \chfn_{\{r \leq R\}} \ell^{-1/2}  \ d\ell dw dr\\
& = & 2\pi \iiint\limits_{\bS(t)} f_0(\mcR(0, t, r, w, \ell), \mcW(0, t, r, w, \ell),\ell) \chfn_{ \{r^2 \leq R^2\} } \ell^{-1/2} \ d\ell dw dr\\
& = & 2\pi \iiint\limits_{\bS(0)} f_0(\tilde{r}, \tilde{w}, \tilde{\ell}) \chfn_{\{\mcR(t, 0, \tilde{r}, \tilde{w}, \tilde{\ell})^2 \leq R^2 \} } \tilde{\ell}^{-1/2}  \ d\tilde{\ell} d\tilde{w} d\tilde{r}
\end{eqnarray*}
where, in the last equality, we have used the change of variables
$$ \left \{
\begin{gathered}
\tilde{r} = \mcR(0, t, r, w, \ell)\\
\tilde{w} = \mcW(0, t, r, w, \ell)\\
\tilde{\ell} = \mcL(0, t, r, w, \ell) = \ell
\end{gathered}
\right. $$
with inverse mapping
$$ \left \{
\begin{gathered}
r = \mcR(t, 0, \tilde{r}, \tilde{w}, \tilde{\ell})\\
w = \mcW(t, 0, \tilde{r}, \tilde{w}, \tilde{\ell})\\
\ell = \mcL(t, 0, \tilde{r}, \tilde{w}, \tilde{\ell}) = \tilde{\ell}
\end{gathered}
\right. $$
and the well-known measure-preserving property (cf. \cite{Glassey}) which guarantees
$$\left \vert \frac{\partial(r, w, \ell)}{\partial (\tilde{r}, \tilde{w}, \tilde{\ell})} \right \vert = 1.$$
Due to Lemma \ref{L1}, we find
$$\mcR(t, 0, \tilde{r}, \tilde{w}, \tilde{\ell})^2 \geq C t^2$$
and thus
$$\{ (\tilde{r}, \tilde{w}, \tilde{\ell}) \in \bS(0) : \mcR(t, 0, \tilde{r}, \tilde{w}, \tilde{\ell})^2 \leq R^2\} \subseteq \{ (\tilde{r}, \tilde{w}, \tilde{\ell}) \in \bS(0) : Ct^2 \leq R^2 \}.$$
Using this produces the upper bound
\begin{eqnarray*}
m(t,R)  & \leq & 2\pi \iiint\limits_{\bS(0)} f_0(\tilde{r}, \tilde{w}, \tilde{\ell})  \chfn_{\{Ct^2 \leq R^2 \} } \tilde{\ell}^{-1/2}  \ d\tilde{\ell} d\tilde{w} d\tilde{r}\\
& = & \mcM  \chfn_{\{Ct^2 \leq R^2 \} }.
\end{eqnarray*}
With this, we have
$$|E(t,x)| = \frac{m(t,r)}{2\pi r} \leq \frac{\mcM}{2\pi r}  \chfn_{\{Ct^2 \leq r^2 \} } \leq Ct^{-1}$$
for every $t > 0, x \in \mathbb{R}^2$, and thus
\begin{equation}
\label{Edecay}
\Vert E(t) \Vert_\infty \leq Ct^{-1}.
\end{equation}

Next, we turn our attention to the lower bound in \eqref{eq:field-bounds} by representing the mass along the largest nontrivial spatial characteristic.
Using \eqref{Vf} and the aforementioned change of variables,
it follows that for any $t \geq 0$
$$\int_0^{\mfR(t)} \int_0^\infty \int_{-\infty}^\infty f(t,r, w, \ell) \ell^{-1/2}\ dw d\ell dr  = \int_0^{\mfR(0)} \int_0^\infty \int_{-\infty}^\infty f_0(\tilde{r}, \tilde{w}, \tilde{\ell}) \tilde{\ell}^{-1/2} \ d\tilde{w} d\tilde{\ell} d\tilde{r}.$$
Inserting the radial charge density into the representation of the enclosed mass and using the above equality, we have
\begin{eqnarray*}
m(t, \mfR(t)) & = & 2\pi \int_0^{\mfR(t)} \int_0^\infty \int_{-\infty}^\infty  f(t, r, w, \ell) \ell^{-1/2}\ dw d\ell dr\\
& = & 2\pi \int_0^{\mfR(0)} \int_0^\infty \int_{-\infty}^\infty f_0(\tilde{r}, \tilde{w}, \tilde{\ell}) \tilde{\ell}^{-1/2} \ d\tilde{w} d\tilde{\ell} d\tilde{r}\\
& = & 2\pi \iiint\displaylimits_{\bS(0)}f_0(\tilde{r}, \tilde{w}, \tilde{\ell}) \tilde{\ell}^{-1/2} \ d\tilde{w} d\tilde{\ell} d\tilde{r}\\
& = & \mcM.
\end{eqnarray*}

Thus, due to the field representation we find for any $t \geq 0$
$$\vert E(t, \mfR(t))\vert = \frac{m(t, \mfR(t))}{2\pi\mfR(t)} = \frac{\mcM}{2\pi \mfR(t)}.$$
Because $f_0$ is nontrivial, we conclude $\mcM \neq 0$.
Finally, since $|E(t,x)|$ attains this value at some $x \in \mathbb{R}^2$, we have
$$\Vert E(t) \Vert_\infty \geq C\mfR(t)^{-1}$$
for $t \geq 0$ and the proof is complete.
\end{proof}

To conclude the estimates implied by the repulsive force, we estimate the field in $L^p(\bfR^2)$ for $2 < p < \infty$ and obtain bounds that will lead to the stated decay rates.

\begin{lemma}
\label{Ep}
For $p \in \left ( 2, \infty \right)$ and solutions of \eqref{VP} or \eqref{RVP},  we have 
$$ \mfR(t)^{-1 + \frac{2}{p}} \lesssim \|E(t) \|_p \lesssim \| E(t) \|_\infty^{1-\frac{2}{p}}$$
\end{lemma}
\begin{proof}
These estimates are similar to those for the three-dimensional problem (see \cite{SP}), but we include them for completeness.
Indeed, we decompose the field integral as
$$\int |E(t,x) |^p \ dx = \int_{|x | < R} |E(t,x)|^p dx + (2\pi)^{1-p} \int_R^\infty \frac{m(t,r)^p}{r^{p}} rdr =: A + B$$
and estimate
$$A \leq 2\pi \|E(t) \|_\infty^p \int_0^R r dr = \pi R^2 \|E(t) \|_\infty^p,$$
while $B$ satisfies
$$B \leq (2\pi)^{1-p} \mcM^p \int_R^\infty r^{1-p} dr \leq CR^{2-p}$$
for $p > 2$.
Optimizing in $R$ yields $R = C\| E(t) \|_\infty^{-1}$ so that
$$\int |E(t,x) |^p \ dx \leq C\| E(t) \|_\infty^{p-2}$$
for any $t \geq 0$.
Raising this to the $\frac{1}{p}$ power yields the stated upper bound.

Next, we prove the lower bound. In particular, using the definition of the field and the maximal spatial support of $f$, we find
$$\int | E(t,x) |^p \ dx = (2\pi)^{1-p} \int_0^\infty m(t,r)^p r^{1-p} \ dr \geq (2\pi)^{1-p} \int_{\mfR(t)}^\infty m(t,r)^p r^{1-p} \ dr.$$
Now, for $r \geq \mfR(t)$, we note that
$m(t,r) = \mcM$ as shown in the proof of Lemma \ref{L2}. Thus, we have
$$ \int_{\mfR(t)}^\infty m(t,r)^p r^{1-p} \ dr = \mcM^p \int_{\mfR(t)}^\infty r^{1-p} \ dr = \frac{\mcM^p}{p-2}\mfR(t)^{2-p}$$
for $p > 2$.
Finally, this implies
$$\int | E(t,x) |^p\ dx \geq C\mfR(t)^{2-p}$$
and hence
$$\|E(t) \|_p \geq C\mfR(t)^{-1+\frac{2}{p}}$$
for any $t >1$.
\end{proof}

Finally, we estimate the behavior of the potential along characteristics and obtain preliminary estimates of characteristics using the field bound.

\begin{lemma}
\label{lem:U-decay}
Along particle characteristics of both \eqref{charang} and \eqref{charang-rel}, the potential $\mcU(t,r)$ satisfies
	\begin{equation*}
	-\mcU(t,\mcR(t))\sim\ln(t)\qquad\text{and}\qquad\|\mcU(t)\|_{\infty}\sim\ln(t).
	\end{equation*}
Furthermore, in the classical case \eqref{charang} we have
$$ \mfW(t) \lesssim \ln(t) \qquad \mathrm{and}  \qquad t \lesssim \mfR(t) \lesssim t\ln(t), $$
while in the relativistic case \eqref{charang-rel} we have
$$ \mfW(t) \lesssim \ln(t) \qquad \mathrm{and} \qquad  \mfR(t) \sim t.$$
\end{lemma}

\begin{proof}
We begin by establishing the maximal position and  {momentum} estimates.
From Lemma \ref{L1}, taking the supremum over the support of $f_0$ immediately yields
$\mfR(t)\gtrsim t$.
Conversely, integrating the characteristic ODEs and using Lemma \ref{L2}, we have
\begin{equation}
\label{eq:mfWrel}
|\mcW(t)| \leq |\mcW(1)| + C\int_1^t s^{-1} \ ds \lesssim \ln(t)
\end{equation}
for the characteristics of both \eqref{VP} and \eqref{RVP}.

Using the  {momentum} bound for the spatial characteristics then yields
	\begin{equation}
	\label{eq:R-upper-bd}
	\mcR(t) \leq \mcR(1) + \int_1^t C(1+ \ln(s) ) \ ds \lesssim t\ln(t)
	\end{equation}
and this further implies
	\begin{equation}
	\label{mfRub}
	\mfR(t) \lesssim t\ln(t).
	\end{equation}
Again, this estimate holds in both the classical and relativistic cases, though due to the relativistic velocity in \eqref{charang} being uniformly upper bounded by $1$,
we can further obtain
$$ \mcR(t) \leq \mcR(1) + \int_1^t 1 \ ds \lesssim t$$
and 
\begin{equation}
\label{eq:mfRrel}
\mfR(t) \lesssim t
\end{equation}
for characteristics of \eqref{RVP}.

With this, we can estimate the potential. Replacing $r$ with $\mcR(t)$ in \eqref{U} and then changing variables (as in the proof of Lemma \ref{L2}) gives
\begin{equation*}
\begin{split}
	-\mcU(t,\mcR(t)) 
	&=
	\frac{1}{2\pi}\int_0^{\mcR(t)} \frac{m(t,q)}{q} \ dq +\int_0^\infty\rho(t,q)q\ln(q)\ dq\\
	&=
	\frac{1}{2\pi}\int_0^{\mcR(t)} \frac{m(t,q)}{q} \ dq +\int_0^\infty\int_{-\infty}^\infty\int_0^\infty f(t,q,w,\ell)\ln(q)\, \ell^{-1/2}\ d \ell\ dw\ dq\\
	&=
	\frac{1}{2\pi}\int_0^{\mcR(t)} \frac{m(t,q)}{q} \ dq +\int_0^\infty\int_{-\infty}^\infty\int_0^\infty f_0(\tilde{q},\tilde{w},\tilde{\ell})\ln \left(\mcR(t,0,\tilde{q},\tilde{w},\tilde{\ell})\right)\, \tilde{\ell}^{-1/2}\ d \tilde\ell\ d\tilde w\ d\tilde q.
\end{split}
\end{equation*}
The nonnegativity of the mass and the lower bound on spatial characteristics from Lemma \ref{L1}, namely $\mcR(t)\gtrsim t$, then give the lower bound
	\begin{align*}
	-\mcU(t,\mcR(t))
	&=
	\frac{1}{2\pi}\int_0^{\mcR(t)} \frac{m(t,q)}{q} \ dq +\int_0^\infty\int_{-\infty}^\infty\int_0^\infty f_0(\tilde{q},\tilde{w},\tilde{\ell})\ln \left(\mcR(t,0,\tilde{q},\tilde{w},\tilde{\ell})\right)\, \tilde{\ell}^{-1/2}\ d \tilde\ell\ d\tilde w\ d\tilde q\\
	&\geq
	0+C\mcM\ln(Ct)\\
	&\gtrsim \ln (t).
	\end{align*}
Next, due to Lemma \ref{L1}, the support of $m(t,r)$ is bounded away from $r =0$ for $t$ sufficiently large, and using  \eqref{eq:R-upper-bd} and \eqref{mfRub} we obtain the upper bound
	\begin{align*}
	-\mcU(t,\mcR(t))
	&=
	\frac{1}{2\pi}\int_0^{\mcR(t)} \frac{m(t,q)}{q} \ dq +\int_0^\infty\int_{-\infty}^\infty\int_0^\infty f_0(\tilde{q},\tilde{w},\tilde{\ell})\ln \left(\mcR(t,0,\tilde{q},\tilde{w},\tilde{\ell})\right)\, \tilde{\ell}^{-1/2}\ d \tilde\ell\ d\tilde w\ d\tilde q\\
	&\leq
	\mcM\ln(\mfR(t))+C\mcM\ln(\mfR(t))\\
	&\leq
	C\ln(t\ln t)\\
	&\lesssim \ln (t).
	\end{align*}
This proves the stated behavior of $-\mcU(t,\mcR(t))$. 
 {Finally, as
$$ \| \mcU(t) \|_\infty \geq -\mcU(t,\mcR(t)) \gtrsim \ln (t)$$
and the upper bound on $-\mcU(t,\mcR(t))$ is uniform in $\mcR(t)$,
it immediately follows that $\|\mcU(t)\|_{\infty}\sim\ln(t)$, and the proof is complete.}
\end{proof}

\section{Energy Estimates}
\label{estimates2}
Now that we have obtained sharp estimates for the behavior of the potential, we can use energy conservation to further refine the growth estimates of momenta in the classical case and obtain lower bounds for both systems.
Here, we treat separately \eqref{VP} and its relativistic counterpart \eqref{RVP}, as the velocity and kinetic energy in these cases are different, and this leads to different rates within the two systems.
For either system, we define
$$T_0 = \sup_{(r,w,\ell) \in S(0)} T(r,w,\ell)$$
where $T(r,w,\ell)$ is the ``turn-around time'' (defined in Lemma \ref{L1}),
and note that $T_0$ is bounded above by a constant that depends only on $S(0)$ due to Lemma \ref{L1} and \eqref{A}.
Hence, taking $t$ sufficiently large implies 
\begin{equation}
\label{Wg1}
\mcW(t,0,r,w,\ell) > 0
\end{equation}
for all $(r,w,\ell) \in S(0)$.

\subsection{The \eqref{VP} System}
We start with \eqref{VP} and its corresponding system of characteristics \eqref{charang}.
\begin{lemma}
\label{Lrefine}
Let $(r,w,\ell) \in S(0)$ be given and let $(\mcR(t), \mcW(t), \ell)$ satisfy \eqref{charang} and \eqref{charanginit} for all $t \geq 0$.
Then, we have
$$\mcW(t,0,r,w,\ell) \lesssim \sqrt{\ln(t)},$$
and $$\mcR(t,0,r,w,\ell) \lesssim t\sqrt{\ln(t)}.$$
Furthermore, the maximal positions and momenta satisfy
$$\mfW(t) \sim \sqrt{\ln(t)},$$
and $$\mfR(t) \sim t\sqrt{\ln(t)}.$$
\end{lemma}

\begin{proof}
To prove the first conclusion, we will use the exchange of energy from potential to kinetic.
In particular, computing an augmented change in energy along particle trajectories, we find
\begin{align*}
	\frac{d}{dt} \Big ( \frac{1}{2} \big (\mcW(t)^2&+ \ell \mcR(t)^{-2} \big )  + \mcU(t,\mcR(t)) \Big )
	=\\
	&=
	\mcW(t)\dot{\mcW}(t)-\ell\mcR(t)^{-3}\mcW(t)+\partial_t\mcU(t,\mcR(t))+\partial_r\mcU(t,\mcR(t))\mcW(t)\\
	&=
	\mcW(t)\left(\ell\mcR(t)^{-3}+\frac{m(t,\mcR(t))}{2\pi\mcR(t)}-\ell\mcR(t)^{-3}-\frac{m(t,\mcR(t))}{2\pi\mcR(t)}\right)+\partial_t\mcU(t,\mcR(t))\\
	&=
	\partial_t\mcU(t,\mcR(t))\\
	&=
	-\int_{\mcR(t)}^\infty \int_{-\infty}^\infty \int_0^\infty w q^{-1} f(t,q,w,\ell) \ell^{-1/2} \ d\ell dw dq.
	\end{align*}
  {The last equality is obtained by taking a time derivative of the expression \eqref{U} of $\mcU$,  using the Vlasov equation \eqref{vlasovang} to eliminate the term $\partial_tf$ and finally integrating by parts in $q$.} Due to \eqref{Wg1}, all momenta on the support of $f(t)$ are positive for sufficiently large times, and it follows that the above derivative is eventually nonpositive.
Integrating for large times gives
 {$$\frac{1}{2} \big (\mcW(t)^2+ \ell \mcR(t)^{-2} \big )  + \mcU(t,\mcR(t)) \leq \frac{1}{2} \big (\mcW(T_0)^2+ \ell \mcR(T_0)^{-2} \big ) + \mcU(T_0,\mcR(T_0))$$
for all $t \geq T_0$.}
Therefore, using Lemma \ref{lem:U-decay} we find
$$ \mcW(t)^2 \lesssim C - \mcU(t,\mcR(t)) \lesssim \ln(t),$$
and the first conclusion follows.
Of course, integrating the upper bound on momenta yields the position estimate
 {$$\mcR(t) \leq \mcR(T_0) + C\int_{T_0}^t \sqrt{\ln(s)} \ ds  \lesssim t\sqrt{\ln(t)}.$$
Further taking the supremum over $(r,w,\ell) \in S(0)$ also yields the upper bounds on the maximal position and momentum}.

Finally, we use energy conservation to obtain the stated lower bounds.
In particular, we find
\begin{equation*}
\begin{split}
\mcE_{\mathrm{VP}} & - \frac{1}{2}\int_0^\infty \int_{-\infty}^\infty \int_0^\infty \mcU(t,r) f(t,r,w,\ell) \ell^{-1/2} \ d\ell dw dr\\
& = \frac{1}{2} \int_0^\infty \int_{-\infty}^\infty \int_0^\infty (w^2 + \ell r^{-2} ) f(t,r,w,\ell) \ell^{-1/2} \ d\ell dw dr \\ 
& = \frac{1}{2} \int_0^\infty \int_{-\infty}^\infty \int_0^\infty (\mcW(t,0,r,w,\ell)^2 + \ell \mcR(t,0,r,w,\ell)^{-2} ) f_0(r,w,\ell) \ell^{-1/2} \ d\ell dw dr \\  
&  {\lesssim \mfW(t)^2 + t^{-2}.}
\end{split}
\end{equation*}
From Lemma \ref{lem:U-decay}, the left side satisfies
\begin{equation*}
\begin{split}
\mcE_{\mathrm{VP}} & - \frac{1}{2}\int_0^\infty \int_{-\infty}^\infty \int_0^\infty \mcU(t,r) f(t,r,w,\ell) \ell^{-1/2} \ d\ell dw dr\\
& = \mcE_{\mathrm{VP}} + \frac{1}{2}\int_0^\infty \int_{-\infty}^\infty \int_0^\infty \Big ( -\mcU(t,\mcR(t,0,r,w,\ell)) \Big ) f_0(r,w,\ell) \ell^{-1/2} \ d\ell dw dr\\
&  {\gtrsim 1 + \ln(t).}
\end{split}
\end{equation*}
Combining these inequalities and taking $t$ sufficiently large yields
$$\mfW(t) \gtrsim \sqrt{\ln(t)}.$$
To obtain the lower bound on positions, we use the virial identity.
In particular, a brief calculation (see \cite[eq. (4.60)]{Glassey}) gives
\begin{equation*}
\begin{split}
& \frac{d^2}{dt^2} \left ( \frac{1}{2} \int_0^\infty \int_{-\infty}^\infty \int_0^\infty r^2 f(t,r,w,\ell) \ell^{-1/2} d \ell dw dr \right )\\
& \qquad  =\int_0^\infty \int_{-\infty}^\infty \int_0^\infty \left ( w^2 + \ell r^{-2} + \frac{1}{2\pi} m(t,r) \right ) f(t,r,w,\ell) \ell^{-1/2} d \ell dw dr.
\end{split}
\end{equation*}
In view of the lower bound on the kinetic energy established above and the nonnegativity of the mass, we find
$$\frac{d^2}{dt^2} \left ( \frac{1}{2} \int_0^\infty \int_{-\infty}^\infty \int_0^\infty r^2 f(t,r,w,\ell) \ell^{-1/2} d \ell dw dr \right ) \gtrsim \ln(t).$$
Lastly, integrating twice gives
$$t^2 \ln(t) \lesssim \int_0^\infty \int_{-\infty}^\infty \int_0^\infty r^2 f(t,r,w,\ell) \ell^{-1/2} d \ell dw dr \lesssim \mfR(t)^2,$$
which proves the final lower bound.
\end{proof}

\subsection{The \eqref{RVP} System}
Now we consider the relativistic system \eqref{RVP}, for which particle velocities are uniformly bounded above by one, and the kinetic energy is first-order, rather than second-order, in the momentum variable. Consequently, we obtain a logarithmic lower bound for the outward  {momentum} $\mcW(t)$.
\begin{lemma}
\label{Lrefine-rel}
The maximal  {momentum} on the support of $f(t)$ satisfies
$$\mfW(t) \gtrsim \ln(t).$$
\end{lemma}

\begin{proof}
As in the proof of Lemma \ref{Lrefine}, we use energy conservation to obtain the result.
In particular, we find
\begin{equation*}
\begin{split}
\mcE_{\mathrm{RVP}} & - \frac{1}{2}\int_0^\infty \int_{-\infty}^\infty \int_0^\infty \mcU(t,r) f(t,r,w,\ell) \ell^{-1/2} \ d\ell dw dr\\
& = \int_0^\infty \int_{-\infty}^\infty \int_0^\infty \sqrt{ 1 + w^2 + \ell r^{-2} } f(t,r,w,\ell) \ell^{-1/2} \ d\ell dw dr \\ 
& =  \int_0^\infty \int_{-\infty}^\infty \int_0^\infty \sqrt{ 1+ \mcW(t,0,r,w,\ell)^2 + \ell \mcR(t,0,r,w,\ell)^{-2} } f_0(r,w,\ell) \ell^{-1/2} \ d\ell dw dr \\  
& \lesssim \sqrt{1 + \mfW(t)^2 + t^{-2}}.
\end{split}
\end{equation*}
From Lemma \ref{U}, the left side satisfies
\begin{equation*}
\begin{split}
\mcE_{\mathrm{RVP}} & - \frac{1}{2}\int_0^\infty \int_{-\infty}^\infty \int_0^\infty \mcU(t,r) f(t,r,w,\ell) \ell^{-1/2} \ d\ell dw dr\\
& = \mcE_{\mathrm{RVP}} + \frac{1}{2}\int_0^\infty \int_{-\infty}^\infty \int_0^\infty \Big ( -\mcU(t,\mcR(t,0,r,w,\ell)) \Big ) f_0(r,w,\ell) \ell^{-1/2} \ d\ell dw dr\\
&  {\gtrsim 1 + \ln(t)}.
\end{split}
\end{equation*}
Combining these inequalities and taking $t$ sufficiently large yields
$$\mfW(t) \gtrsim \ln(t)$$
as desired.
\end{proof}

\section{Estimates of the Charge Density}
\label{density}

We first address upper bounds for the \eqref{VP} and \eqref{RVP} systems.

\begin{lemma}
\label{L5}
The solutions of \eqref{VP} and \eqref{RVP} both satisfy
$$\Vert \rho(t) \Vert_\infty \lesssim t^{-1}.$$
\end{lemma}

\begin{proof}
Our strategy is similar to that of \cite{Horst}, and we use backwards characteristics to estimate the size of the $w$ support of $f(t,r,w,\ell)$ for fixed $r, \ell > 0$.
 {As the estimates may depend more sensitively on time, we may employ the generic constant $C > 0$ in some places, rather than the ``$\lesssim$'' notation.}
First, consider solutions of \eqref{VP}.
Note that due to Lemma \ref{L2} and the time-reversibility of characteristics, we have
\begin{equation}
\label{fielddecay}
\frac{m(\tau, \mcR(\tau, t, r, w, \ell))} {2\pi\mcR(\tau, t, r, w, \ell)} + \ell\mcR(\tau, t, r, w, \ell)^{-3}  
 \leq C\tau^{-1}
\end{equation}
for any $\tau \geq 2$  and $(r,w,\ell) \in \bS(t)$.
Let $(r, w_1, \ell), (r, w_2, \ell) \in \bS(t)$ be given.
Then, integrating the characteristic equations \eqref{charang}
we find for $t \geq 2$ and $k=1,2$
$$\mcR(2, t, r, w_k, \ell) = r - w_k (t-2) + \int_2^t \int_s^t \left ( \frac{m(\tau, \mcR(\tau))}{2\pi\mcR(\tau)} + \ell\mcR(\tau)^{-3} \right )  \ d\tau ds.$$
Hence, subtracting these expressions and using \eqref{fielddecay} yields for $t$ sufficiently large
\begin{eqnarray*}
\left |\mcR(2,t, r, w_1, \ell) - \mcR(2,t, r, w_2, \ell) \right |
& \geq & \left | w_1 - w_2 \right | (t-2) - C\int_2^t \int_s^t \tau^{-1} \ d\tau ds\\
& \geq & \left | w_1 - w_2 \right | (t-2) - Ct.
\end{eqnarray*}
Due to the global existence result, $f(2,r,w,\ell)$ is compactly supported and we note that
$$\left |\mcR(2,t, r, w_1, \ell) - \mcR(2,t,r,w_2, \ell) \right | \leq |\mcR(2,t, r, w_1, \ell) | + | \mcR(2,t,r,w_2, \ell) | \leq C.$$
Rearranging the inequality above then produces
$$|w_1 - w_2| \lesssim 1.$$
Therefore, the diameter of the  {momentum} support is uniformly bounded.
This implies that for any $t$ sufficiently large and fixed $r, \ell > 0$, there is $C > 0$ and $w_0 \in \mathbb{R}$ such that
\begin{equation}
\label{supportsubset}
\left \{ w : f(t,r, w, \ell) \neq 0 \right \} \subseteq \left \{ w \in \mathbb{R} : | w - w_0 | \leq C \right \}.
\end{equation}

To obtain a similar estimate in the case of the relativistic system \eqref{RVP}, we merely repeat the steps of this argument, using the same field bound.
A straightforward calculation as above then
allows us to estimate the difference between differing velocities for $t$ sufficiently large, namely
$$\left |\frac{w_1}{\sqrt{1 +w_1^2 +\ell r^{-2}}} - \frac{w_2}{\sqrt{1 +w_2^2 +\ell r^{-2}}} \right| \lesssim 1$$
upon integrating the field bound.
Finally, because the derivative of $w \mapsto \frac{w}{\sqrt{1 +w^2 +\ell r^{-2}}}$ is positive and uniformly bounded below on the support of $f(t)$, this bound also holds for the difference of momenta.
%
%
%
Hence, we again arrive at \eqref{supportsubset} in the relativistic case.

Next, we use the spherical representation of $\rho(t,r)$ to complete the estimate for solutions of \eqref{VP}.
Let $r > 0$ be given. Note that if $r \not\in \pi_r(S(t))$, then $f(t,r,w,\ell) = 0$ for all $w \in \bfR$ and $\ell > 0$, and thus $\rho(t,r) = 0$.
Alternatively, if $r \in \pi_r(S(t))$ then there exists $(\tilde{r}, \tilde{w}, \tilde{\ell}) \in S(0)$ such that
$$r = \mcR(t, 0, \tilde{r}, \tilde{w}, \tilde{\ell}).$$
By Lemma \ref{L1}, we find
$$r^{-1}  = \mcR(t, 0, \tilde{r}, \tilde{w}, \tilde{\ell})^{-1} \lesssim t^{-1}. $$
Using this along with the invariance of $\ell$ along characteristics, the assumption \eqref{A} on $\bS(0)$, and \eqref{supportsubset}, we have
\begin{eqnarray*}
\rho(t,r) & = & \frac{1}{r} \int_0^\infty \int_{-\infty}^{\infty} f(t,r,w,\ell)  \ell^{-1/2}\ dw d\ell\\
& \leq & C t^{-1} \int_{C_1}^{C_2} \int_{-\infty}^\infty f(t,r, w, \ell)  \ell^{-1/2} \ dw d\ell\\
& \leq & C \|f_0 \|_\infty t^{-1}  \int_{C_1}^{C_2} \left | \left \{ w : f(t,r, w, \ell) \neq 0 \right \} \right |  \ell^{-1/2}\ d \ell \\
& \leq & C t^{-1}
\end{eqnarray*}
for $r \in \pi_r(S(t))$ and $t$ large.
Combining this with the case $r \not\in \pi_r(S(t))$ and taking the supremum then yields
\begin{equation}
\label{rhodecay}
\|\rho(t)\|_\infty \lesssim t^{-1}
\end{equation}
for either \eqref{VP} or \eqref{RVP}.
\end{proof}

Finally, a lower bound on the supremum of the charge density follows trivially.
\begin{lemma}
\label{rhobelow}
There is $C > 0$ such that solutions of \eqref{VP} and \eqref{RVP} satisfy
$$\Vert \rho(t) \Vert_\infty \geq C\mfR(t)^{-2}$$
for any $t \geq 0$.
\end{lemma}
\begin{proof}
Using the enclosed mass, we find for any $t \geq 0$
$$\mcM = m(t, \mfR(t)) = 2\pi \int_0^{\mfR(t)} q \rho(t,q) \ dq
\leq C \Vert \rho(t) \Vert_\infty \mfR(t)^2.$$
Rearranging this inequality then yields the result.
\end{proof}

\section{Proof of Theorems}
\label{thmproof}

\begin{proof}[Proof of Theorem \ref{T1}]

To obtain the estimates stated in Theorem \ref{T1}, we merely collect results of the lemmas.
In particular, Lemma \ref{Lrefine} yields the sharp asymptotic behavior of $\mfR(t)$ and $\mfW(t)$, while combining the upper bounds of this lemma with
Lemma \ref{L1} provide the stated pointwise estimates on characteristics.
The behavior of the potential is directly implied by Lemma \ref{lem:U-decay}.
Using the upper bound on positions from Lemma \ref{Lrefine}, namely
\begin{equation}
\label{eq:mfR}
\mfR(t) \lesssim t\sqrt{\ln(t)},
\end{equation}
within Lemma \ref{L2} gives the upper and lower bounds on $\|E(t) \|_\infty$
and further inserting these estimates into Lemma \ref{Ep} provides
the upper and lower bounds on $\|E(t) \|_p$ for any $p \in (2,\infty)$.
Finally, Lemma \ref{L5} gives the upper bound on the charge density, and inserting \eqref{eq:mfR} into the result of Lemma \ref{rhobelow} gives the lower bound.
\end{proof}

\begin{proof}[Proof of Theorem \ref{T2}]
As in the previous proof, we merely collect results of the lemmas.
In particular, upper bounds on $\mfW(t)$ and $\mfR(t)$ are obtained from \eqref{eq:mfWrel} and \eqref{eq:mfRrel}, respectively, while the lower bound on $\mfR(t)$ follows from $\mcR(t) \gtrsim t$ in Lemma \ref{L1} and that for $\mfW(t)$ is given by Lemma \ref{Lrefine-rel}.
Combining these upper bounds with the results of Lemma \ref{L1} provide the stated pointwise estimates on characteristics.
As in the previous proof, the behavior of the potential is directly implied by Lemma \ref{lem:U-decay} and the remaining asymptotic behavior for the field and charge density follows by using Lemmas \ref{L2}, \ref{Ep}, \ref{L5}, and \ref{rhobelow} with the estimate
$\mfR(t) \lesssim t$ to provide the necessary lower bounds.
\end{proof}

\appendix
\section{Derivation of the Radially-Symmetric Expressions}\label{appendix}

In the appendix we demonstrate how to change variables from integrals in Cartesian coordinates in $\R^4$ to the radially-symmetric variables and justify the forms of the charge density, potential {, electric field, and energy}. First, we consider a function $\phi : \mathbb{R}^4 \to \mathbb{R}$ of the form
$$\phi(x,v) = \phi (r, w, \ell)$$
whose dependence can be represented exactly in terms of the radial coordinates
$$r = \vert x \vert, \qquad w = \frac{x \cdot v}{r}, \qquad \ell = \vert x \wedge v \vert^2.$$

To compute the $v$-integral of this function, we first note that we can, without loss of generality, rotate a given vector $x \in \mathbb{R}^2$ so that it points in the $v_1$ direction. In particular, we express such a vector as $x = [r ,0]^T$ as $\vert x \vert = r$ and rewrite
\begin{eqnarray*}
\int \phi(x,v) \ dv &= & \iint \phi \left (\vert x \vert, \frac{x \cdot v}{\vert x \vert}, \vert x \wedge v \vert^2 \right) \ dv_1 dv_2\\
& = & \iint \phi(r, v_1, r^2v_2^2 ) \ dv_1 dv_2.
\end{eqnarray*}

Because the integrand is even in $v_2$, we find
$$ \iint \phi(r, v_1, r^2v_2^2 ) \ dv_1 dv_2 = 2\int_{-\infty}^\infty \int_0^\infty \phi(r, v_1, r^2v_2^2 ) \ dv_2 dv_1.$$
Next, we change variables so that
$$\left \{ \begin{gathered}
a = v_1\\
b = r^2v_2^2
\end{gathered} \right.$$
or
$$\left \{ \begin{gathered}
v_1 = a\\
v_2 = r^{-1} b^{1/2}
\end{gathered} \right.$$
so that
$\frac{dv_2}{db} = \frac{1}{2} r^{-1}b^{-1/2}$
and find
$$2\int_{-\infty}^\infty \int_0^\infty \phi(r, v_1, r^2v_2^2 ) \ dv_2 dv_1 = \int_{-\infty}^\infty \int_0^\infty r^{-1} b^{-1/2}\phi(r, a, b) \ db da.$$

Finally, relabeling the variables of integration yields
$$\int \phi(x,v) \ dv = \int_{-\infty}^\infty \int_0^\infty r^{-1} \ell^{-1/2}\phi(r, w, \ell) \ d\ell dw,$$
and in particular,
$$\rho(t,x) = \int f(t,x,v) \ dv = r^{-1} \int_{-\infty}^\infty \int_0^\infty \ell^{-1/2}f(t, r, w, \ell) \ d\ell dw$$
so that $\rho(t,x)$ can be expressed uniquely in terms of the radial spatial variable as $\rho(t,r)$.
Furthermore, the enclosed mass can be expressed using radial coordinates as
$$m(t,x) = \int_{\vert y \vert \leq \vert x \vert} \rho(t,y) \ dy = \int_0^{2\pi} \int_0^r q \rho(t, q) \ dq d \theta,$$
which shows that $m$ also depends only upon the radial variable and simplifies to
$$m(t,r) = 2\pi \int_0^r\int_{-\infty}^\infty \int_0^\infty \ell^{-1/2}f(t,  {q}, w, \ell) \ d\ell dw d {q}.$$

Next, we derive the stated formula for the potential. In particular, because the charge density is radial and we have
$$\mcU(t,x) = -\frac{1}{2\pi} \ln(|x|) \star_x \rho(t, |x|),$$
we find that $\mcU = \mcU(t, |x|)$ is radial, as it is the convolution of radial functions.
Furthermore, using polar coordinates the above formula implies
$$\mcU(t,0) = -\frac{1}{2\pi} \int \ln(|y|) \rho(t,|y|) \ dy = -\int_0^\infty q \ln(q) \rho(t,q) \ dq.$$
As the potential is radial, the electric field $E(t,x)$ points in the outward radial direction due to the relationship $E(t,x) = -\nabla_x \mcU(t,x)$ so that
$$E(t,x) = \mathcal{E}(t,r) \frac{x}{r},$$
where $\mathcal{E}$ is determined  by the Divergence Theorem.
In particular, we have
$$m(t,r) = \int_{|x| \leq r} \rho(t,|x|) \ dx = \int_{|x| \leq r} \nabla_x \cdot E(t,x) \ dx = \int_{|x| = r} E(t,x) \cdot n \ dS = 2\pi r \mathcal{E}(t,r),$$
which implies
$$E(t,x) = \frac{m(t,r)}{2\pi r} \frac{x}{r}.$$
With this, the potential must satisfy
$$-\partial_r \mcU(t,r) = \frac{m(t,r)}{2\pi r}.$$
Thus, integrating and using the formula for $\mcU(t,0)$ computed above gives
$$\mcU(t,r) = -\frac{1}{2\pi} \int_0^r \frac{m(t,q)}{q} dq  - \int_0^\infty q \ln(q) \rho(t,q) \ dq.$$

 {Finally, the energy of either system can be derived in a straightforward manner using the radial coordinates.
In particular, as
	\begin{equation}\label{eq:expression-for-v}
	|v|^2 = |x \cdot v|^2 + |x \wedge v|^2 = w^2 + \ell r^{-2},
	\end{equation}
we can write the energy for  \eqref{VP} as
\begin{eqnarray*}
\mcE_{\mathrm{VP}} & = &\frac{1}{2} \iint |v|^2 f(t,x,v) \ dv dx + \iint \mcU(t,x) f(t,x,v) \ dv dx\\
& = & \frac{1}{2} \int_0^\infty \int_{-\infty}^\infty \int_0^\infty (w^2 + \ell r^{-2} ) f(t,r,w,\ell) \ell^{-1/2} \ d\ell dw dr\\
& \ & \quad  + \frac{1}{2}\int_0^\infty \int_{-\infty}^\infty \int_0^\infty \mcU(t,r) f(t,r,w,\ell) \ell^{-1/2} \ d\ell dw dr
\end{eqnarray*}
with an analogous representation for $\mcE_{\mathrm{RVP}}$, as stated in the introductory section.}

\end{document}